\documentclass{cimart}

\usepackage{enumerate}

\title{% Please, capitalize only the first word
   A nonlinear analogue of additive commutators
    }

\authors{% Please, use "Firstname Lastname" format, without abreviations
Truong Huu Dung, Tran Nam Son, Pham Duy Vinh
    }

\authorinfo[% Format: please, use "F. Name"
    T. Huu Dung]{% Format: please, use "Department (only if 
    Dong Nai University, %9 Le Quy Don Str., Tam Hiep Ward, Dong Nai Province, 
    Vietnam}{%
    thdung@dnpu.edu.vn or dungth0406@gmail.com
    }

\authorinfo[% Format: please, use "S. Name"
    T. Nam Son]{% Format: please, use "Department (only if 
    University of Science %, Ho Chi Minh City, Vietnam; 
    and Vietnam National University, %Ho Chi Minh City, 
    Vietnam}{%
    trannamson1999@gmail.com
    }

\authorinfo[% Format: please, use "T. Name"
    P. Duy Vinh]{% Format: please, use "Department (only if 
     University of Science, %Ho Chi Minh City, 
     Vietnam National University, %Ho Chi Minh City, 
     and Dong Nai University, %9 Le Quy Don Str., Tam Hiep Ward, Dong Nai Province, 
     Vietnam}{%
    phamduyvinh.dlu@gmail.com
    }

\abstract{%
We study a nonlinear analogue of additive commutators, known as \textit{polynomial commutators}, defined by \( p(ab) - p(ba) \) for a polynomial \( p \in F[x] \) and elements \( a, b \) in an algebra \( R \) over a field \( F \). Originally introduced by Laffey and West for matrices over fields, this notion is here extended to broader algebraic settings. We first show that in division rings, polynomial commutators can generate maximal subfields and even the entire ring as an algebra. In the matrix setting, we prove that matrices similar to ones with zero diagonal are polynomial commutators, and under mild assumptions, every matrix can be written as a product of at most three such commutators. Furthermore, we demonstrate that the matrix algebra can be decomposed as the sum of its center and the linear span of all polynomial commutators. Using the theory of rational identities in division rings, we also exhibit that the trace of a polynomial commutator in the matrix ring can be nonzero in noncommutative cases. Lastly, we explore the size of polynomial commutators via matrix norms.				
    }



\acknowledgments{First and foremost, the authors would like to express their sincere gratitude to the editor for the prompt handling of the manuscript and to the referees for the numerous valuable comments and suggestions, which have greatly improved the quality of this paper. They are also grateful to the Mathematics Stack Exchange community, particularly the contributor to Question 5062751, for insightful discussions. Finally, the second author (Tran Nam Son) is funded by Vietnam National University of Ho Chi Minh City (VNU-HCM) under grant number T2025-18-02.}
    
\VOLUME{34}
\YEAR{2026}
\ISSUE{1}
\NUMBER{9}
\DOI{https://doi.org/10.46298/cm.16053}
\licence{CC BY 4.0}
\editinfo{July 16, 2025}{March 13, 2026}{Utkir Rozikov}

\begin{document}

	\section{Introduction}\label{intro}

In an associative algebra \( R \) over a field $F$, expressions of the form \( ab-ba \) for \( a, b \in R \) are known as \emph{additive commutators}, which quantify the extent to which multiplication in \( R \) fails to be commutative. This naturally leads to a  nonlinear perspective that one may study expressions such as \( p(ab) - p(ba) \), where \( p \in F[x] \) is a polynomial in the variable $x$ with coefficients in a field \( F \). We implicitly assume that \( p \) has no constant term when the algebra $R$ is non-unital, ensuring that \( p(ab) \) and \( p(ba) \) are still meaningful. When \( p \) is linear, this reduces to the classical commutator \( ab - ba \), but for nonlinear polynomials, such expressions capture more subtle interactions between noncommutativity and polynomial structure. 	It turns out that this idea, originally introduced by T.~J.~Laffey and T.~T.~West \cite{Pa_LaWe_93} in the context of matrix algebras over fields, motivates our investigation. The first part of this paper extends their framework to a more general setting --- specifically, to certain rings that are not necessarily commutative. From this starting point, the remainder of the paper follows a natural line of development, as detailed below.

To set the stage, we now fix an associative algebra \( R \) over a field \( F \) and let \( p\in F[x] \) be a nonconstant polynomial in the variable \( x \) over \( F \). A \textit{polynomial commutator} in \( R \) relative to $p$ is defined as any element of the form \( p(ab) - p(ba) \) for some \( a, b \in R \). It is worth noting that when the algebra $R$ is a non-unital algebra, we tacitly restrict our attention to the  polynomial \( p \) without constant term, so that expressions like \( p(ab) \) and \( p(ba) \) make sense without relying on the presence of the multiplicative identity. In the special case where \( p(x) = x \), the polynomial commutator \( p(ab) - p(ba) \) reduces to the familiar form \( ab - ba \), commonly referred to as an \emph{additive commutator}, or simply a \emph{commutator}.
For clarity and brevity, we introduce the bracket notation: \begin{itemize}
	\item \([a,b] = ab - ba\) for the (additive) commutator of $a$ and $b$ in $R$,
	\item  and  
	\(p[a,b] = p(ab) - p(ba)\) for the polynomial commutator of $a$ and $b$ in $R$ relative to $p$.
\end{itemize} We then write  
\begin{itemize}
	\item \([R, R] = \{ [a,b] \mid a,b \in R \}\) for the set of all commutators in $R$,
	\item  and  
	\(p[R, R] = \{ p[a,b] \mid a,b \in R \}\) for the corresponding set of polynomial commutators in $R$.
\end{itemize}It is well known that the additive commutator $ab-ba$ measures how far two elements $a$ and $b$ in $R$ deviate from commuting, so we are led to the following question. 

\begin{question}\label{qS}
	Let $R$ be an associative noncommutative algebra over a field $F$. Can one find two noncommuting elements $a$ and $b$ in $R$ such that $p(ab)\neq p(ba)$?
\end{question} In other words, Question~\ref{qS} asks whether, in a noncommutative algebra, whether the intersection $p[R,R]\cap(R\setminus\{0\})$ is nonempty. To investigate Question~\ref{qS}, observe that for any \( a, b \in R \), a straightforward substitution that if $$p(x) =a_0+ a_1x + a_2x^2 + \dots + a_mx^m,$$ where $m\geq1$ is an integer and  $a_0,a_1,a_2,\ldots,a_m\in F$, then
\begin{equation}
	p[a,b] = a_1[a,b] + a_2((ab)^2 - (ba)^2) + \ldots + a_m((ab)^m - (ba)^m).\label{sub}
\end{equation} Clearly, if $ab=ba$, then every term vanishes and $p[a,b]=0$. To settle our question, we consider the sequence
$(ab)^i - (ba)^i$ for $i\geq1$.
In the setting of a division ring, the powerful machinery of rational identities guarantees that one can choose a pair of noncommuting elements \(a,b\) for which
$(ab)^i - (ba)^i \neq 0$ for all $i\ge1.$
Therefore, in every division ring, the answer to Question~\ref{qS} is affirmative, as we will establish in Section~\ref{section division}. We are grateful to the Mathematics Stack Exchange community, particularly the contributor to Question 5062751, for insightful comments that helped shape this answer.

Returning to expression~\eqref{sub}, 	to analyze the difference \( p[a,b]=p(ab) - p(ba) \), we can apply a standard telescoping identity valid in any (possibly noncommutative) ring. For elements \( X, Y \) and any positive integer \( i \), one has
\[
X^i - Y^i = \sum_{k=0}^{i-1} X^k (X - Y) Y^{i-1-k}.
\]	Taking \( X = ab \) and \( Y = ba \), this gives
\[
(ab)^i - (ba)^i = \sum_{k=0}^{i-1} (ab)^k (ab - ba)(ba)^{i-1-k} = \sum_{k=0}^{i-1} (ab)^k [a, b] (ba)^{i-1-k},
\]
where \( [a, b] = ab - ba \) denotes the additive commutator. Substituting the expression \eqref{sub} for \( (ab)^i - (ba)^i \), we obtain
\begin{equation}\label{te}
	p[a,b] = \sum_{i=1}^{n} a_i \sum_{k=0}^{i-1} (ab)^k [a, b] (ba)^{i-1-k},
\end{equation}as a linear combination of elements obtained by sandwiching the commutator \( [a, b] \) between powers of $ab$ and $ba$, with coefficients depending on \( p \) and the powers of \( ab \) and \( ba \). This expansion \eqref{te} is central to our analysis and allows us to deduce Theorem~\ref{po} directly. We omit the proof, as it follows immediately from the definition of the commutator ideal, namely, the two-sided ideal generated by all additive commutators of the form \( [a, b] = ab - ba \) with \( a, b \in R \). In particular, every element in the commutator ideal is a finite sum of expressions of the form \( a [c, d] b \), where \( a, b, c, d \in R \).

\begin{theorem}\label{po}
	Let $R$ be an associative algebra over a field \( F \) and let \( p\in F[x] \) be a nonconstant polynomial. Then, the commutator ideal of $R$ contains $p[R,R]$.
\end{theorem}

Continuing along the lines of Theorem~\ref{po}, we recall from \cite[Theorem 3.4]{Pa_Ga_25} that if a unital associative algebra \( R \) over a field \( F \) coincides with its commutator ideal, then \( R \) is generated by commutators as an \( F \)-algebra. Moreover, there exists a positive integer \( N \) such that for every element \( a \in R \), there are elements \( b_j, c_j, d_j, e_j \in R \), for \( j = 1, \ldots, N \), satisfying
\[
a = \sum_{j=1}^N [b_j, c_j][d_j, e_j].
\]	With this result in mind, it is natural to ask whether an analogous result holds in the context of \emph{polynomial commutators}, namely
\begin{question}\label{commutator ideal}
	Does there exist a unital associative $F$-algebra \(R\) such that the two-sided ideal of \(R\) generated by $p[R,R]$
	coincides with \(R\), where $p\in F[x]$ is a nonconstant polynomial? If so, must \(R\) then be generated as an \(F\)-algebra by $p[R,R]$? Furthermore, in that case, is there a positive integer \(N\) for which every element \(\alpha\in R\) can be expressed as
	\[
	\alpha =\sum_{j=1}^N p[a_j, b_j]p[c_j, d_j]	\]
	for suitable \(a_j,b_j,c_j,d_j\in R\)?	
\end{question} To address Question~\ref{commutator ideal}, we continue to work within the framework of division rings. For the first two parts of the question, we establish affirmative results in the case where \( R \) is a division ring. For the final part, we focus on a specific instance: the division algebra of real quaternions $\mathbb{H}$, as detailed in Section~\ref{section division}. One of the central steps in addressing the final part of Question~\ref{commutator ideal} is to understand the structure of the set \( p[\mathbb{H}, \mathbb{H}] \). Hence, we now turn our attention to describing \( p[R, R] \), beginning with the case where \( R \) is a matrix ring.

In the case of matrices over fields, by exploiting the commutativity of the base field along with a result of T.~J.~Laffey and T.~T.~West \cite{Pa_LaWe_93}, it is shown that the set of all polynomial commutators coincides exactly with the set of trace-zero matrices. In particular,	working over the matrix ring \( R = \mathrm{M}_n(F) \), where \( F \) is a field of characteristic zero and \( n > 1 \), it was shown in \cite{Pa_LaWe_93} that $$\{p(ab) - p(ba) \mid a, b \in R\} \supseteq \{a \in R \mid \mathrm{trace}(a) = 0\}.$$ On the other hand, this inclusion follows directly from the linearity and cyclic invariance of the trace function, namely $ \mathrm{trace}(ab) = \mathrm{trace}(ba) $ and  $\mathrm{trace}(\lambda a + \mu b) = \lambda \, \mathrm{trace}(a) + \mu \, \mathrm{trace}(b) $ for all scalars \( \lambda, \mu \in F \). 	Therefore, we obtain the equality
\[
\{p(ab) - p(ba) \mid a, b \in R\} = \{a \in R \mid \mathrm{trace}(a) = 0\}.
\]
Furthermore, a classical theorem of Shoda, Albert, and Muckenhoupt \cite{Pa_Sho_36, Pa_Al_57} asserts that every trace-zero matrix is a commutator, that is,
\[
\{a \in R \mid \mathrm{trace}(a) = 0\} = \{ab - ba \mid a, b \in R\}.
\]
Hence, we conclude $p[R,R]=[R,R]$, proving that the nonlinear case reduces	to the linear case. This raises a natural and intriguing question: 
\begin{question}\label{que}
	Does it hold that  $p[R,R]=[R,R]$?
\end{question}

To address Question~\ref{que}, we begin by revisiting a key component of the argument from~\cite{Pa_LaWe_93}, focusing first on whether the trace of a polynomial commutator is necessarily zero. As demonstrated in Theorem~\ref{zz}, this expectation is false. Despite the challenges encountered in the matrix case, we are able to achieve a complete characterization of polynomial commutators in a closely related setting - the real quaternion algebra $\mathbb{H}$, thereby answering Question~\ref{que} in the case  $R=\mathbb{H}$, as shown in Theorem~\ref{qua}. This result is particularly striking and can be seen as a noncommutative analogue of the trace-zero matrix phenomenon.

Guided by Theorem~\ref{zz},  we consider a weaker form of Question~\ref{que}: 
\begin{question}\label{que1}
	If $R=\mathrm{M}_n(A)$ where $n\geq2$ is an integer and $A$ is a noncommutative associative $F$-algebra, does $p[R,R] \supseteq \{a \in R \mid \mathrm{trace}(a) = 0\}$?
\end{question}
To explore Question~\ref{que1}, we revisit~\cite{Pa_LaWe_93}, where over fields, every noncentral traceless matrix is similar to one with zero diagonal, a property extending to positive characteristic fields~\cite[Proposition~8]{Pa_AmRo_94}. For \( R = \mathrm{M}_n(D) \) with a noncommutative division ring \( D \), it is known~\cite[Proposition 1.8]{Pa_AmRo_94} that any noncentral matrix is similar to one with diagonal entries zero except possibly one \( d \in D \), but it remains unclear whether \( d \) can always be chosen as zero when the matrix is traceless. Thus, we focus on matrices with zero diagonal entries in Section~\ref{matrix}.

Following  Theorem~\ref{qua} again, it is known that \(p[\mathbb{H},\mathbb{H}]\) is naturally an \(\mathbb{R}\)-vector space, so we now turn to the analogous question:

\begin{question}\label{what}
	What algebraic structures does \(p[R,R]\) carry?
\end{question}

Question~\ref{what} will be explored in Section~\ref{section what}, where we investigate the structures generated by \( p[R, R] \), such as semigroups and linear spans. Other algebraic structures, like ideals, subrings, or subalgebras generated by \( p[R, R] \), may also emerge from this analysis. However, we will not consider the subgroup generated by \( p[R, R] \), as this would require invertibility of the elements involved.

In situations where one cannot give an explicit description of every polynomial commutator, a natural alternative is to gauge them by size. To that end, we work in the setting of normed vector spaces, specifically for the algebra of complex \( n \times n \) matrices equipped with some norm. Although many of the ideas we develop may extend to the broader context of Banach algebras, focusing on \( \mathrm{M}_n(\mathbb{C}) \) makes the discussion more concrete and intuitively clear. Suppose \( V \) and \( W \) are normed vector spaces over the same field (either \( \mathbb{R} \) or \( \mathbb{C} \)), and let \( A: V \to W \) be a linear map. Recall that  such an \( A \) is continuous exactly when there exists a real number \( c \ge 0 \) so that
$\|Av\|_W \le c\,\|v\|_V,$ for every $v \in V.$
In plain terms, \( A \) never stretches a vector by more than a factor of \( c \).  This quantity measures the maximum amount by which \( A \) can lengthen any vector.

Within this framework, it is natural to ask:
\begin{question}\label{estimate}
	Given a nonconstant polynomial \( p \in \mathbb{C}[x] \) and matrices \( A, B \in \mathrm{M}_n(\mathbb{C}) \), does there exist a real number \( c\geq0 \) such that
	$\left\| p(AB) - p(BA) \right\| \le c \left\| A \right\|\cdot\| B \| ?$
\end{question}Recall in \cite[Theorem 4.1, page 225]{Pa_BoWe_05} that a key motivating estimate is the Böttcher-Wenzel inequality, which asserts that for all \( A, B \in \mathrm{M}_n(\mathbb{C}) \), the Frobenius norm of their commutator satisfies
$\|AB - BA\|_F \le \sqrt{2} \, \|A\|_F \, \|B\|_F.$ In connection with Question~\ref{estimate}, we therefore pose the following question:

\begin{question}\label{estimate1}
	For any nonconstant polynomial \( p \in \mathbb{C}[x] \) and matrices \( A, B \in \mathrm{M}_n(\mathbb{C}) \), does there exist a real number \( c \ge 0 \) such that $\left\| p(AB) - p(BA) \right\| \le c \left\| AB - BA \right\|?$
\end{question}

In Section~\ref{how big}, we tackle Questions~\ref{estimate} and~\ref{estimate1} by deriving three fundamental tools: a Frobenius‐norm bound, a numerical‐radius bound, and an integral representation for the norm of a polynomial commutator. These results give concrete size estimates for polynomial commutators in \( \mathrm{M}_n(\mathbb{C}) \). 

To summarize, the paper is organized as follows. Section~\ref{section division} focuses on Questions~\ref{qS}, and~\ref{commutator ideal}, examining it in the context of division rings. Section~\ref{matrix} addresses Questions~\ref{que},  and~\ref{que1}, which pertain to matrix algebras. In Section~\ref{section what}, we turn to Question~\ref{what}, where the case of matrix algebras and AW*-algebras is considered. Finally, Section~\ref{how big} explores Questions~\ref{estimate} and~\ref{estimate1}, concerning norm estimates for polynomial commutators.

\section{Division algebras}\label{section division}

This section is devoted to addressing Question~\ref{qS} in the setting of division algebras. We begin with a few simple but useful observations. The conclusion is immediate when \( p \) is a nonconstant linear polynomial. Indeed, since \( R \) is noncommutative, we can always choose noncommuting elements \( a, b \in R \), and then  $p[a, b] = a_1[a, b] \ne 0,$
where \( a_1 \ne 0 \) because \( p \) is nonconstant and linear, and \( [a, b] \ne 0 \) by choice.  

We now turn to the case where \( p \) is  quadratic. Consider \( R = \mathbb{H} \), the division ring of real quaternions with standard basis \( \{1, i, j, k\} \). Take \( a = i \), \( b = j \), so that \( ab = k \) and \( ba = -k \). Then we compute: $$p[a, b] = p(k) - p(-k) = a_2(k^2 - (-k)^2) + a_1(k - (-k)) = 2a_1k,$$
where \( a_1, a_2 \in \mathbb{R} \), and \( a_2 \ne 0 \) since \( p \) is quadratic. In this case, \( p[a, b] \ne 0 \) whenever \( a_1 \ne 0 \).  To cover the remaining possibility where \( a_1 = 0 \), we choose \( a = i \), \( b = i + j \). Then,
\[
ab = i(i + j) = -1 + k, \quad ba = (i + j)i = -1 - k,
\]
and so, $$p[a, b] = a_2((-1 + k)^2 - (-1 - k)^2) = a_2((-2k) - (2k)) = -4a_2k \ne 0.$$ Therefore, even when \( p \) is purely quadratic, we still have \( p(ab) \ne p(ba) \). This highlights the relevance of understanding the expressions \( (ab)^i - (ba)^i \) in such contexts, as mentioned in the following theorem.

\begin{theorem}\label{division}
	If $R$ is a noncommutative division algebra over a field $F$, then for any positive integer \( n \), there exist elements \( \alpha, \beta \in R \) such that $(\alpha\beta)^n \neq (\beta\alpha)^n.$
\end{theorem}

\begin{proof}
	Assume, for contradiction, that \((\alpha\beta)^n = (\beta\alpha)^n\) for all \(\alpha, \beta \in R\).
	Given nonzero elements \(x\) and \(y\) in \(R\), substituting \(\alpha = x\) and \(\beta = yx^{-1}\) yields
	$xy^{n}x^{-1} = y^{n},$
	and hence \(xy^{n} = y^{n}x\).
	Thus the \(n\)-th power of every element is central.
	However, this contradicts a classical theorem of Kaplansky \cite[Theorem]{Pa_Ka_51}, which asserts that such a relation forces \(R\) to be commutative.
	Therefore, there must exist \(\alpha, \beta \in R\) such that \((\alpha\beta)^n \neq (\beta\alpha)^n\).
\end{proof}

With Theorem~\ref{division} established, we are now prepared to confirm the following result, which provides an affirmative answer to Question~\ref{qS} in the context of division algebras. Although Theorem~\ref{division} will play a role in the proof, we emphasize that the assumption that the base field \( F \) is infinite is not required here. Moreover, Theorem~\ref{division} may be regarded as a special case of the following theorem for the polynomial \( p(x) = x^n \), whereas the result below holds for any nonconstant polynomial and does not rely on the infinitude of \( F \).

\begin{theorem}\label{division1}
	Let \( R \) be a noncommutative division $F$-algebra, and let \( p \in F[x] \) be a nonconstant polynomial. Then there exist elements \( \alpha, \beta \in R \) such that \( p(\alpha\beta) \ne p(\beta\alpha) \).
\end{theorem}

\begin{proof}
	Suppose, for contradiction, that \(p(ab) = p(ba)\) for all \(a,b \in R\).
	This is equivalent to requiring that the noncommutative polynomial $p(x_1x_2) - p(x_2x_1),$
	defined in the free algebra over \(F\) in the noncommuting variables \(x_1\) and \(x_2\), vanishes under all substitutions \(x_1 \mapsto a\), \(x_2 \mapsto b\) with \(a,b \in R\).
	By a celebrated theorem of Kaplansky~\cite[Theorem~1]{Pa_Ka_48}, it follows that \(R\) is finite-dimensional over \(F\). Thus \(R \cong F^n\), where \(n=\dim_F R\), and multiplication $\mu : F^n \times F^n \to F^n, \mu(a,b)=ab,$
	is a bilinear map and hence is given coordinate-wise by polynomial functions that are homogeneous of degree \(2\).
	In particular, for any \(d\ge 1\), the map $(a,b) \mapsto (ab)^d - (ba)^d$
	is given coordinate-wise by homogeneous polynomials of degree \(2d\), and it is nonzero by Theorem~\ref{division}.
	Hence, in the identity $$p(ab)-p(ba)=\sum_{d=1}^{k} c_d \big( (ab)^d - (ba)^d \big),$$
	where \(\displaystyle p(x)=\sum_{d=0}^k c_d x^d\) with \(c_k\neq 0\), the summands are linearly independent because they are homogeneous polynomial functions of distinct degrees. Since the left-hand side vanishes for all \(a,b\), linear independence forces \(c_d=0\) for all \(d>0\); hence \(p\) is constant, a contradiction.
	Here we implicitly use that \(F\) is infinite, which is always the case for the center of a noncommutative finite-dimensional division algebra by Wedderburn’s ``Little'' Theorem~\cite[(13.1), Chapter~5, \S13, p.~203]{Bo_Lam_01}.
	The contrapositive of the above argument shows that if \(p\in F[x]\) is nonconstant, then \(p(ab)-p(ba)\) is not identically zero on \(R\).
\end{proof}

Theorem~\ref{division1} also shows that for any nonconstant polynomial, there always exists a nonzero polynomial commutator in a noncommutative division ring. In other words, if all polynomial commutators vanish for every nonconstant polynomial, then the ring must be commutative, reducing the situation to a trivial case.

We now turn to Question~\ref{commutator ideal}. The following theorem provides an affirmative answer for the first two parts of the question in the case where $R$ is a division ring.

\begin{theorem}\label{gener}

	If $R$ is a noncommutative finite-dimensional division algebra over a field $F$ and $p\in F[x]$ is a nonconstant polynomial, then $R$ is generated by $p[R,R]$ as an $F$-algebra.
\end{theorem}

The proof of Theorem~\ref{gener} is carried out in two main steps. In Step 1, we establish the existence of a maximal subfield of \( R \) that is generated by  $p[R,R]$. Step 2 then invokes a result of I.~N.~Herstein and A.~Ramer, specifically \cite[Corollary 2]{Pa_He_72}.

We begin with Step 1 by recalling that a \emph{maximal subfield} of a division ring \( R \) is a subfield that is not properly contained in any other subfield of \( R \). These subfields are essential in the structural analysis of division rings, and in particular, the finite-dimensionality of \( R \) often reflects that of its maximal subfields. By Zorn’s Lemma, every division ring admits at least one maximal subfield. 

To study such subfields in our context, we adopt a method inspired by the work of \cite{Pa_Aa_18}, which is based on rational identities in central simple algebras. However, we restrict our attention here to the special case of division algebras, rather than working in the more general setting.

Let \( F\langle\mathcal{X}\rangle \) denote the unital associative \( F \)-algebra of noncommutative polynomials in the set $\mathcal{X}=\{X_1, X_2, \ldots\}$ of variables.  A noncommutative polynomial \( f \in F\langle\mathcal{X}\rangle\setminus\{0\}  \) is said to be a \textit{polynomial identity} of a unital associative \( F \)-algebra \( R \) if it evaluates to zero for all permissible substitutions from \( R \).  A central example, which will play a key role in our arguments, is the following. Given a positive integer \( n \) and noncommuting indeterminates \( y_0, y_1, \dots, y_n \), define  
\[
g_n(y_0, y_1, \dots, y_n) = \sum_{\delta \in S_{n+1}} \operatorname{sign}(\delta) \, y_0^{\delta(0)} y_1 y_0^{\delta(1)} y_2 y_0^{\delta(2)}  \dots y_n y_0^{\delta(n)},
\]
where \( S_{n+1} \) denotes the symmetric group on \( \{0, 1, \dots, n\} \) and \( \operatorname{sign}(\delta) \) is the sign of the permutation \( \delta \). This polynomial, introduced in \cite{Bo_Be_96}, encodes algebraicity conditions in terms of identities involving \( n+1 \) noncommuting variables.

We now state the following lemma, which is a direct consequence of \cite[Corollary 2.3.8]{Bo_Be_96}.

\begin{lemma} \label{algebraic}
	Let $R$ be a finite-dimensional division $F$-algebra and let $n,m\geq1$ be integers. For any element \( a \in \mathrm{M}_n(R) \), the following statements are equivalent:
	\begin{enumerate}[\rm (i)]
		\item The element \( a \) is algebraic over the center of $R$  with degree less than or equal to \( m \).
		\item The polynomial identity $g_m(a, r_1, \dots, r_m) = 0$
		holds for all \( r_1, \dots, r_m \in \mathrm{M}_n(R) \).
	\end{enumerate}
\end{lemma}

Recall that an element \( a \) in a ring \( R \) with center \( Z = Z(R) \) is said to be \textit{algebraic of degree} \( n \) over \( Z \) if there exists a polynomial \( f(x) \in Z[x] \) of degree \( n \) such that \( f(a) = 0 \), and no polynomial of smaller degree satisfies this condition. Note that \( f(x) \) need not be irreducible, even when \( Z \) is a field.

Let \( \mathcal{I}(A) \) denote the set of all  polynomial identities of a central simple algebra \( A \). The following result relates the identity structure of such algebras and follows directly from \cite[Theorem 11]{Pa_Am_66}. It is known from \cite[(15.8) Theorem, p.~242]{Bo_Lam_01} that the dimension of such division rings is always a perfect square. This insight leads us to consider the square root of the dimension as a natural quantity in the statement of the next result.

\begin{lemma} \label{identity}
	Let $R$ be a finite-dimensional division $F$-algebra and let $n=\sqrt{\dim_FR}$. If \( L \) is an extension field of \( F \), then $\mathcal{I}(R) = \mathcal{I}(\mathrm{M}_n(F)) = \mathcal{I}(\mathrm{M}_n(L)).$
\end{lemma}

In what follows, we continue with the following lemma that ensures applying $g_n$ to a nonzero polynomial.

\begin{lemma} \label{l2.3}
	Let $F$ be an infinite field, and let \( p(x) \in F[x] \) be a nonconstant polynomial. Then,  $g_n(p(x_1x_2) - p(x_2x_1), x_3, \dots, x_{n+2})$	is a nonzero polynomial in the noncommuting variables \( x_1, \dots, x_{n+2} \).
\end{lemma}

\begin{proof}
	First, we construct a division ring $D$ such that $g_n$ has a nonzero value over $D$. Let $G$ be a free group generated by infinitely many independent indeterminates $\{t_i: i \in I\}$. By \cite[Theorem (14.21)]{Bo_Lam_01}, the Mal'cev-Neumann ring $D = F((G))$ is a division ring. The infinite set $\{t_i: i \in I\}$ is linearly independent over the center $Z(D)$ of $D$. Since $F \subseteq Z(D)$, it follows that  $\dim_F D = \infty$. Next, suppose, for the sake of contradiction, that for all \( a_1, a_2 \in D \), the element \( p(a_1a_2) - p(a_2a_1) \) is algebraic over \( F \) of degree at most \( n \). Then, taking \cite[Theorem 1.2]{Pa_Hai.Dung.Bien_2022} into account, it would follow that $\dim_F D < \infty$, contradicting the fact that $\dim_F D = \infty$.	Therefore, there exist elements \( a_1, a_2 \in D \) such that \( p(a_1a_2) - p(a_2a_1) \) is either transcendental over \( F \) or algebraic of degree strictly greater than \( n \). It follows from Lemma~\ref{algebraic}  that
	$g_n(p(a_1a_2) - p(a_2a_1), a_3, \dots, a_{n+2}) \ne 0$
	for some \( a_3, \dots, a_{n+2} \in D \). Hence, the polynomial \( g_n(p(x_1x_2) - p(x_2x_1), x_3, \dots, x_{n+2}) \) must be nonzero. This completes the proof.
\end{proof}

To explore the existence of maximal subfields within division rings, we rely on the following lemma, which is an immediate consequence of \cite[Corollary 15.6 and Proposition~15.7]{Bo_Lam_01}.

\begin{lemma}  \label{maximal}
	Let $D$ be a finite-dimensional division $F$-algebra and let $n=\sqrt{\dim_FD}$. If \( K \) is a subfield of \( D \) containing \( F \), then \( \dim_F K \leq n \), and equality holds if and only if \( K \) is a maximal subfield of \( D \).  
\end{lemma}

We continue to present the following lemma that identifies algebraic elements whose degree over the center equals $\sqrt{\dim_FD}$. The proof follows from \cite[Lemma 5]{Pa_Aa_18}, with a slight modification.

\begin{lemma}\label{element}
	Let \( F \) be an infinite field and let \( n \geq 2 \) be an integer. If \( p \in F[x] \) is a nonconstant polynomial, then there exist matrices \( A, B \in \mathrm{M}_n(F) \) such that \( p[A,B] \) is an algebraic element of degree \( n \) over \( F \).
\end{lemma}

\begin{proof}
	It is shown in \cite[Theorem]{Pa_LaWe_93} that any noncentral matrix in $\mathrm{M}_n(F)$ with trace zero can be expressed as a polynomial commutator in $\mathrm{M}_n(F)$ relative to \(p\). Although \cite[Theorem]{Pa_LaWe_93} was originally stated for fields of characteristic 0, the result also holds for noncentral matrices over fields with positive characteristic. A proof for this case will be provided in Corollary~\ref{pooo}.
	
	Consider a noncentral \( n \times n \) matrix \( T = (t_{ij}) \in \mathrm{M}_n(F) \), where
	\[
	t_{ij} = 
	\begin{cases} 
		1, & \text{if } j = i + 1, \\
		0, & \text{otherwise}.
	\end{cases}
	\]The matrix \( T \) has trace zero and is algebraic of degree \( n \). This completes the proof.
\end{proof}

Having established the necessary groundwork, we are now prepared to present the following theorem, which confirms the existence of a maximal subfield and thereby completes Step 1 as anticipated.

\begin{theorem}\label{ss}
	Let \( R \) be a noncommutative finite-dimensional  division \( F \)-algebra, and let \( p \in F[x] \) be a nonconstant polynomial. Then, there exist elements \( a, b \in R \) such that the  subdivision $F$-algebra of \( R \) generated by \( p[a,b] \) is a maximal subfield of \( R \).
\end{theorem}

\begin{proof}
	It is known from Wedderburn’s ``Little'' Theorem~\cite[(13.1), Chapter~5, \S13, p.~203]{Bo_Lam_01} that the center of a noncommutative finite-dimensional division algebra must be infinite, so  \(F\) is infinite. Let \( n = \sqrt{\dim_F R} \). For any elements $a,b\in R$, we write  \( F(p[a,b]) \) to denote the intersection of all subdivision rings of \( R \) containing both \( F \) and \( p[a,b] \).  It is straightforward to verify that this intersection forms a field. By Lemma~\ref{maximal}, it suffices to find elements \( a, b \in R \) such that  $\dim_F F(p[a,b]) \geq n$.  Define  $$\ell = \max \{ \dim_F F(p[a,b]) \mid a, b \in R \}.$$ 	Then, by Lemma~\ref{algebraic}, for any \( a, b \in R \) and all \( r_1, \dots, r_\ell \in R \), we have $$g_\ell(p[a,b], r_1, \dots, r_\ell) = 0.$$  		Hence, the polynomial  $g_\ell(p[x_{\ell+1},x_{\ell+2}], x_1, \dots, x_\ell)$ 
	is a polynomial identity of \( R \) in noncommuting variables \( x_1, \dots, x_{\ell+2} \). Taking Lemma~\ref{l2.3} into account, this polynomial is nonzero. By Lemma~\ref{identity}, this identity also holds in the matrix ring \( \mathrm{M}_n(F) \). Hence, it follows that $g_\ell(p[A,B], A_1, \dots, A_\ell) = 0$ for all $A,B,A_1, \dots, A_\ell \in \mathrm{M}_n(F)$. Moreover, by Lemma~\ref{algebraic}, each element of the form \( p[A,B] \), with \( A, B \in \mathrm{M}_n(F) \), is algebraic over \( F \) with degree at most \( \ell \).  On the other hand, by Lemma~\ref{element}, we can find matrices \( A, B \in \mathrm{M}_n(F) \) such that \( p[A,B] \) is algebraic over \( F \) with degree \( n \). Therefore, $n\leq\ell$, as claimed.
\end{proof}

We now make use of a result due to I. N. Herstein and A. Ramer, presented in \cite[Corollary~2]{Pa_He_72}, and stated below for convenience.

\begin{lemma}\label{ma}
	Let \( R \) be a finite-dimensional division \( F \)-algebra. If \( K \) is a maximal subfield of \( R \) generated by some element \( \alpha \in R \), then \( R \) is generated as an \( F \)-algebra by \( \alpha \) and a conjugate of \( \alpha \).
\end{lemma}

With Lemma~\ref{ma} and the preceding tools in hand, we are ready to prove Theorem~\ref{gener}.

\begin{proof}[The proof of Theorem~\ref{gener}]
	By Theorem~\ref{ss}, there exist elements \(\alpha, \beta \in R\) such that the field generated by \(p[\alpha,\beta]\) over \(F\), denoted \(F(p[\alpha,\beta])\), is a maximal subfield of \(D\). Note that this field is the intersection of all subdivision rings of \(R\) that contain both \(F\) and the element \(p[\alpha,\beta]\).  Applying Lemma~\ref{ma}, we deduce that \(R\) is generated as an \(F\)-algebra by \(p[\alpha,\beta]\) and one of its conjugates. Furthermore, this conjugate is also a polynomial commutator. Specifically, if the conjugate is given by \(\gamma p[\alpha,\beta] \gamma^{-1}\) for some \(\gamma \in R\setminus\{0\}\), then we can express it as
	$\gamma p[\alpha,\beta] \gamma^{-1}= p[\gamma\alpha\gamma^{-1},\gamma\beta\gamma^{-1}],$ which is again of the same polynomial commutator form. This completes the proof.
\end{proof}

As mentioned in Section~\ref{intro}, we now continue with an affirmative answer to the final part of Question~\ref{commutator ideal}. Our focus is on a concrete example: the division algebra of real quaternions. As a first step, we describe the set \( p[\mathbb{H}, \mathbb{H}] \).

Let \( R = \mathbb{H} \), the real quaternion algebra with standard basis \( \{1, i, j, k\} \) and the following multiplication rules: $ i^2 = j^2 = k^2 = ijk = -1.$ Any element \( q \in \mathbb{H} \) can be written as \( q = a + bi + cj + dk \), where \( a, b, c, d \in \mathbb{R} \), with \( \Re(q) = a \) representing the real part and \( \Im(q) = bi + cj + dk \) representing the imaginary part. A quaternion \( q \) is called \emph{purely imaginary} if \( \Re(q) = 0 \), i.e., \( q = bi + cj + dk \). For any quaternion \( q = a + bi + cj + dk \), where \( a, b, c, d \in \mathbb{R} \), 
\begin{enumerate}[\rm (i)]
	\item the conjugate of \( q \), denoted \( \overline{q} \), is given by
	$\overline{q} = a - bi - cj - dk$;
	\item the norm of $q$, denoted $\|q\|$, is defined as $\|q\|=\sqrt{q\overline{q}}$.
\end{enumerate}
With the quaternionic framework in place, we now state the following theorem, which describes the set of all polynomial commutators associated to a given polynomial $p$ in $\mathbb{H}$.

\begin{theorem}\label{qua}
	If \(p\in\mathbb{R}[x]\) is any nonconstant real polynomial in the quaternionic variable $x$, then 
	\(p[\mathbb{H},\mathbb{H}]=\{bi + cj + dk \mid b,c,d\in\mathbb R\}\).
\end{theorem}

To prove Theorem~\ref{qua}, we first establish a few preliminary lemmas.

\begin{lemma}\label{re}
	If $\alpha,\beta\in\mathbb{H}$ and $\lambda,\gamma\in\mathbb{R}$, then
	\begin{enumerate}[\rm (i)]
		\item $\Re(\alpha\beta)=\Re(\beta\alpha)$;
		\item $\Re(\lambda\alpha\pm\gamma\beta)=\lambda\Re(\alpha)\pm\gamma\Re(\beta)$.
	\end{enumerate}	 
\end{lemma}

\begin{proof}
	Assume that $\alpha = x + yi + zj + tk$ and $\beta = e + fi + gj + hk$ are elements of $\mathbb{H}$ where $x, y, z, t, e, f, g, h \in \mathbb{R}$. Then, a straightforward computation shows that \[ \Re(\alpha\beta) = xe - yf - zg - th = \Re(\beta\alpha),\] which proves (i). For (ii), it is not difficult to verify that \[ \Re(a\alpha \pm b\beta) = ax \pm be = a\Re(\alpha) \pm b\Re(\beta), \] for all $a, b \in \mathbb{R}$, as promised.
\end{proof}

\begin{lemma}\label{ree}
	If $\alpha,\beta\in\mathbb{H}$, then   $\Re((\alpha\beta)^n)=\Re((\beta\alpha)^n)$ for all non-negative integers $n$.
\end{lemma}

\begin{proof}
	By Lemma~\ref{re}, $\Re((\alpha\beta)^n)=\Re(\alpha\beta(\alpha\beta)^{n-1})=\Re(\beta(\alpha\beta)^{n-1}\alpha)=\Re((\beta\alpha)^{n})$.   
\end{proof}

We are now ready to prove Theorem~\ref{qua}.

\begin{proof}[The proof of Theorem~{\rm\ref{qua}}]
	We divide the proof into two parts. The first part is to show that $p[\mathbb{H},\mathbb{H}] \subseteq \{\, bi + cj + dk \mid b,c,d \in \mathbb{R} \,\}.$ For any $\alpha,\beta \in \mathbb{H}$, we can write
	\[
	p[\alpha,\beta] = \sum_{\ell=0}^{n} c_\ell\bigl((\alpha\beta)^\ell - (\beta\alpha)^\ell\bigr).
	\]
	Hence, by Lemmas~\ref{re} and~\ref{ree}, $$\Re(p[\alpha,\beta]) = \sum_{\ell=0}^{n} c_\ell \bigl(\Re((\alpha\beta)^\ell) - \Re((\beta\alpha)^\ell)\bigr) = 0,$$
	as claimed.
	
	The second part is to show that $p[\mathbb{H},\mathbb{H}] \supseteq \{\, bi + cj + dk \mid b,c,d \in \mathbb{R} \,\}.$
	Let $v\in\mathbb{H}$ such that $\Re(v)=0$.		The case $v = 0$ is trivial, so we assume $v \neq 0$. We shall construct $a,b \in \mathbb{H}$ such that $p(ab) - p(ba) = v$. The argument breaks into three steps.
	
	\medskip
	\noindent\textbf{Step 1. Reduction to odd powers.}  We first consider the following two cases.
	
	\textit{Case 1: $p$ has no odd-degree terms.} Suppose $\displaystyle p(x) = \sum_{m=0}^{\frac{n}{2}} c_{2m} x^{2m},$ with each $c_{2m} \in \mathbb{R}.$
	Fix a real number $s \neq 0$ and let $w$ be a purely imaginary quaternion. Then
	\begin{align*}
		(s + w)^{2m} - (s - w)^{2m}
		&= \sum_{\ell=0}^{2m} \binom{2m}{\ell} s^{2m - \ell} w^\ell
		- \sum_{\ell=0}^{2m} \binom{2m}{\ell} s^{2m - \ell} (-1)^\ell w^\ell \notag\\
		&= \sum_{\ell=0}^{2m} \binom{2m}{\ell} s^{2m - \ell} \bigl(1 - (-1)^\ell\bigr) w^\ell \notag\\
		&= 2 \sum_{\substack{\ell = 1 \\ \ell \text{ odd}}}^{2m}
		\binom{2m}{\ell} s^{2m - \ell} w^\ell. 
	\end{align*} Hence,
	\begin{align*}
		p(s + w) - p(s - w)
		&= \sum_{m=0}^{\frac{n}{2}} c_{2m} \bigl[(s + w)^{2m} - (s - w)^{2m}\bigr] \notag\\
		&= 2 \sum_{m=0}^{\frac{n}{2}} c_{2m}
		\sum_{\substack{\ell = 1 \\ \ell \text{ odd}}}^{2m}
		\binom{2m}{\ell} s^{2m - \ell} w^\ell. 
	\end{align*} Since $w$ is purely imaginary, it follows that  $w^2 = -\|w\|^2 \in \mathbb{R}$. Consequently,
	\[
	w^{2\ell} = (-1)^\ell \|w\|^{2\ell}, \qquad
	w^{2\ell + 1} = (-1)^\ell \|w\|^{2\ell} w.
	\]
	Therefore, $$p(s + w) - p(s - w)
	= 2 H_1(s, \|w\|^2) w,$$
	where
	\[
	H_1(s, \|w\|^2)
	:= \sum_{m=0}^{\frac{n}{2}} c_{2m}
	\sum_{\substack{\ell = 1 \\ \ell \text{ odd}}}^{2m}
	\binom{2m}{\ell} s^{2m - \ell}
	(-1)^{\frac{\ell-1}{2}} (\|w\|^2)^{\frac{\ell-1}{2}}.
	\]
	
	\medskip
	\textit{Case 2: $p$ has at least one odd-degree term.} For any purely imaginary $w$,
	\[
	p(w) - p(-w)
	= \sum_{\ell=0}^{n} c_\ell \bigl(w^\ell - (-w)^\ell\bigr)
	= 2 \sum_{\substack{\ell = 1 \\ \ell \text{ odd}}}^{n} c_\ell w^\ell.
	\]
	Since $w^2 = -\|w\|^2$ is real, it follows that $w^{2m+1} = (-1)^m \|w\|^{2m} w$. Hence,
	\[
	p(w) - p(-w)
	= 2 \Biggl(\sum_{m=0}^{\lfloor \frac{n-1}{2} \rfloor}
	c_{2m+1} (-1)^m \|w\|^{2m}\Biggr) w
	= 2 H_2(\|w\|^2) w,
	\]
	where
	\[
	H_2(t) = \sum_{m=0}^{\lfloor \frac{n-1}{2} \rfloor} c_{2m+1} (-1)^m t^m
	\]
	is a real polynomial in $t$.
	
	\medskip
	\noindent\textbf{Step 2. Choosing $a$ and $b$.} Introduce a real parameter $t > 0$ (to be determined later) and set $w = t \dfrac{v}{\|v\|}.$
	
	\textit{Case 1: $p$ has no odd-degree terms.}
	Let $\alpha := s + w$ and $\beta := s - w$.
	Then there exists $b \in \mathbb{H}^\times$ such that $b \alpha b^{-1} = \beta$.
	Indeed, if $w = 0$, we may take $b = 1$.
	If $w \neq 0$, then it follows from  \cite[Theorem~2.2]{Pa_Zha_97} that 
	there exists a quaternion $b$ such that $b w b^{-1} = -w$.
	Consequently, $$b \alpha b^{-1} = b(s + w)b^{-1} = s + b w b^{-1} = s - w = \beta.$$
	Setting $a := \alpha b^{-1}$, we then have $ab = \alpha$ and $ba = \beta$.
	
	\medskip
	\textit{Case 2: $p$ has at least one odd-degree term.}
	Choose a unit imaginary quaternion $u$ orthogonal to $v$, that is,
	$\|u\| = 1$, $u^2 = -1$, $u^{-1} = -u$, and $u \cdot v = 0$.
	Then $u \cdot w = 0$, so $uw = -wu$.
	Set $b = u$ and $a = b w$.
	It is easy to verify that
	\[
	ab = (b w)b = (u w)u = u(wu) = u(-u w) = (-u^2)w = w,
	\]
	and
	\[
	ba = b(b w) = u(u w) = u^2 w = -w.
	\]
	
	\medskip
	\noindent\textbf{Step 3. Solving for $t$.} Define a real function $h : \mathbb{R} \to \mathbb{R}$ by
	\[
	h(t) =
	\begin{cases}
		H_1(1, t^2), & \text{if } p \text{ has no odd-degree terms},\\[4pt]
		H_2(t^2), & \text{if } p \text{ has at least one odd-degree term}.
	\end{cases}
	\]
	It follows from Step~1 that
	\[
	p[a,b]
	= 2 h(\|w\|^2) w
	= 2 h(t^2) \Bigl(t \frac{v}{\|v\|}\Bigr)
	= \bigl(2t\,h(t^2)\bigr)\frac{v}{\|v\|}.
	\]
	Requiring this expression to equal $v$ gives the scalar equation
	\[
	2t\,h(t^2) = \|v\|.
	\]
	The left-hand side is a continuous odd function in $t$, vanishing at $t = 0$ and unbounded as $|t| \to \infty$. Hence, by the Intermediate Value Theorem, there exists a unique positive solution $t$. With this $t$ in mind, the elements $a,b$ constructed above satisfy $p[a,b] = v$, as promised.
\end{proof}

As promised, we are now ready to address the final part of Question~\ref{commutator ideal}, establishing the following theorem with a stronger conclusion that confirms the case  $N=1$.

\begin{theorem}
	If $p\in\mathbb{R}[x]$ is a nonconstant polynomial, then every element of $\mathbb{H}$ can be expressed as a product of two elements in $p[\mathbb{H},\mathbb{H}]$.
\end{theorem}

\begin{proof}
	Let \(\alpha \in \mathbb{H}\) and write \(\alpha = a_1 + a_2 i + a_3 j + a_4 k\) for some \(a_1, a_2, a_3, a_4 \in \mathbb{R}\). By \cite[Lemma~2.1]{Pa_Zha_97}, there exists a nonzero element \(\gamma \in \mathbb{H}\) such that  
	\[
	\gamma^{-1} \alpha \gamma = a_1 + \sqrt{a_2^2 + a_3^2 + a_4^2} \, i.
	\]
	On the other hand, we observe that  
	\[
	a_1 + \sqrt{a_2^2 + a_3^2 + a_4^2} \, i = j  \left( -a_1 j + \sqrt{a_2^2 + a_3^2 + a_4^2} \  k \right).
	\]
	According to Theorem~\ref{qua}, both elements \(j\) and \(-a_1 j + \sqrt{a_2^2 + a_3^2 + a_4^2}  k\) belong to \(p[\mathbb{H}, \mathbb{H}]\). Following the proof of Theorem~\ref{gener} again, it is known that any conjugate of a polynomial commutator is also a polynomial commutator.  It follows that $$\alpha=\gamma j\gamma^{-1}  \cdot \gamma\left( -a_1 j + \sqrt{a_2^2 + a_3^2 + a_4^2} \ k \right)\gamma^{-1}$$ can be expressed as a product of two elements in \(p[\mathbb{H}, \mathbb{H}]\), as desired.
\end{proof}

\section{Matrix algebras}\label{matrix}

In Section~\ref{intro}, we were led to Questions~\ref{que} and~\ref{que1}. 	To proceed, let us revisit a key part of the argument in \cite{Pa_LaWe_93}, starting with the question of whether the trace of a polynomial commutator is always zero. One might first expect this to be true, especially in the setting of matrix rings over noncommutative division rings, which resemble matrix rings over fields in many structural aspects. However, this expectation turns out to be incorrect, as shown in the following theorem.

\begin{theorem}\label{zz}
	Let $R$ be a  noncommutative division $F$-algebra and let $n\geq2$ be an integer. If $p\in F[x]$ is a nonconstant polynomial, then $$p[\mathrm{M}_n(R),\mathrm{M}_n(R)]\cap\{A\in\mathrm{M}_n(R)\mid\mathrm{trace}(A)\neq0\}\neq\emptyset.$$
\end{theorem}

\begin{proof}
	By Theorem~\ref{division1}, there exist elements \( \alpha, \beta \in R \) such that \( p(\alpha\beta) \ne p(\beta\alpha) \).		Consider the diagonal matrices \( a = \mathrm{diag}(1, \ldots, 1, \alpha) \) and \( b = \mathrm{diag}(1, \ldots, 1, \beta) \) in \( \mathrm{M}_n(R) \). A direct computation shows that   $\mathrm{trace}(p(ab) - p(ba)) = p(\alpha\beta) - p(\beta\alpha).$  Since \( p(\alpha\beta) \ne p(\beta\alpha) \), the trace is nonzero. This confirms that \( p[\mathrm{M}_n(R), \mathrm{M}_n(R)] \) contains a matrix with nonzero trace, as promised.
\end{proof}

With Theorem~\ref{zz} in mind, we are now led to consider Question~\ref{que1}. We revisit part of the reasoning in \cite{Pa_LaWe_93}. In their setting, they used the fact that over a field of characteristic zero, any noncentral traceless matrix is similar to one with only zeros along the main diagonal. This property also extends to fields of positive characteristic, as inferred directly from \cite[Proposition 8]{Pa_AmRo_94}.

However, the situation becomes more delicate when \( R = \mathrm{M}_n(D) \), where \( D \) is a noncommutative division ring and $n>1$ is an integer. Although such rings share many structural similarities with fields, it remains unknown whether every noncentral matrix in \( \mathrm{M}_n(D) \) is similar to a matrix with all diagonal entries equal to zero. A related but weaker result is known: by \cite[Proposition 1.8]{Pa_AmRo_94}, any noncentral matrix in \( \mathrm{M}_n(D) \) is similar to a matrix whose diagonal entries are all zero except possibly one entry \( d \in D \). Importantly, the proof of \cite[Proposition~1.8]{Pa_AmRo_94} does not require the global assumption that \( D \) is finite-dimensional over its center. This result holds in full generality, regardless of the dimensionality of \( D \) over \( F \). This raises a natural question: can we always choose \( d = 0 \) in this similarity? That is, is every noncentral traceless matrix similar to a matrix with zero diagonal? Though this seems plausible and is stated without proof in \cite[Proposition 9]{Pa_Me_13}, there appears to be no direct derivation of this claim from \cite[Proposition 1.8]{Pa_AmRo_94}. We are grateful to Z. Mesyan \cite{Pa_Me_13} for clarifying that a small oversight occurred in his reference to Amitsur and Rowen's result, though it does not affect his conclusions, which concern matrices over fields.

If \( D \) is commutative, then the claim follows directly from the fact that trace is invariant under similarity. This means that any similarity transformation preserves the trace of a matrix, so if a noncentral matrix is traceless, then it must be similar to one with all diagonal entries equal to zero. In contrast to the field case, similarity does not preserve trace over noncommutative division rings. For example, let \( D = \mathbb{H} \) be the real quaternion division ring with standard basis \( \{1, i, j, k\} \) satisfying \( i^2 = j^2 = k^2 = ijk = -1 \). Consider the similarity:
\[
\begin{pmatrix}
	j & 0 \\ i & 1
\end{pmatrix}^{-1}
\begin{pmatrix}
	i & j \\ -j & i
\end{pmatrix}
\begin{pmatrix}
	j & 0 \\ i & 1
\end{pmatrix}
=
\begin{pmatrix}
	0 & 1 \\ 0 & 0
\end{pmatrix}.
\]
Here, the original matrix has nonzero trace but is similar to a traceless matrix. More generally, if \( a, b \in D \) do not commute, then the traceless diagonal matrix $\begin{pmatrix}
	a & 0 \\ 0 & -a
\end{pmatrix}$
is similar to a matrix with nonzero trace:
\[
\begin{pmatrix}
	1 & 0 \\ 0 & b
\end{pmatrix}
\begin{pmatrix}
	a & 0 \\ 0 & -a
\end{pmatrix}
\begin{pmatrix}
	1 & 0 \\ 0 & b
\end{pmatrix}^{-1}=\begin{pmatrix}
a&0\\
0&-bab^{-1}
\end{pmatrix}.
\]
This shows that trace is not preserved under similarity when \( D \) is noncommutative. Even if we attempt to adapt the method of \cite[Proposition 1.8]{Pa_AmRo_94}, difficulties arise quickly. Consider the simple case \( n = 2 \) and a noncentral traceless matrix of the form
\[
A = \begin{pmatrix} a & b \\ c & -a \end{pmatrix} \in \mathrm{M}_2(D).
\]
As explained in \cite[Proposition 1.8]{Pa_AmRo_94}, for any \( u \in D \), one can compute
\[
\begin{pmatrix} 1 & u \\ 0 & 1 \end{pmatrix} A \begin{pmatrix} 1 & u \\ 0 & 1 \end{pmatrix}^{-1} = 
\begin{pmatrix} a + uc & -(a + uc)u + (b - ua) \\ c & -cu - a \end{pmatrix}.
\]
To obtain zero diagonal, we would need \( a + uc = 0 \) and \( -cu - a = 0 \), which together imply \( u = -ac^{-1} = -c^{-1}a \), a condition that requires \( ac = ca \). Since this does not hold in general, the argument fails.

For this reason, we will work initially with matrices whose diagonals are all zero. 

To proceed, we require a series of auxiliary lemmas. The first lemma can be proven by induction on the matrix size, in a manner similar to the argument in \cite[Lemma 2.1]{Pa_DuSo_25}. For conciseness, we omit the details.

\begin{lemma}\label{cheo hoa}
	Let \( R \) be a unital associative algebra over an infinite field \( F \), and let \( n \) be a positive integer. Suppose \( A \in \mathrm{M}_n(R) \) is a triangular matrix whose diagonal entries \( a_1, a_2, \ldots, a_n \) lie in \( F \) and are pairwise distinct. Then, there exists an invertible matrix \( P \in \mathrm{M}_n(R) \) such that  $P^{-1} A P = \mathrm{diag}(a_1, \ldots, a_n).$
\end{lemma}

We now state a standard result, but include its proof here for completeness.

\begin{lemma}\label{distinct}
	Let \( F \) be an infinite field, and let \( p \in F[x] \) be a nonconstant polynomial. Then the image of \( p \) evaluated on \( F \), namely \( \{p(a) \mid a \in F\} \), is infinite.
\end{lemma}

\begin{proof}
	Define the evaluation map \( \varphi_p:F\to F \) by \(\varphi_p(a)=p(a)\) for all \(a\in F\). Suppose for contradiction that the image of \( \varphi_p \) is finite, say \( \{a_1, \ldots, a_k\} \), with \( a_i \neq a_j \) for \( i \neq j \), where $k\geq1$ is an integer. Then,  $\displaystyle F = \bigcup_{i=1}^k \varphi_p^{-1}(a_i).$ Since \( F \) is infinite, at least one of the fibers \( \varphi_p^{-1}(a_i) \) must be infinite. Let this be \( \varphi_p^{-1}(a_{i_0}) \). Then the equation \( p(x) = a_{i_0} \) has infinitely many solutions in \( F \), contradicting the fact that \( p \) is a nonconstant polynomial over a field. This contradiction proves the lemma.
\end{proof}

With these lemmas established, we can now prove the following main result.

\begin{theorem}\label{zero}
	Let \( R \) be a unital associative algebra over an infinite field \( F \), and let \( n > 1 \) be an integer. Suppose that $p\in F[x]$ is a nonconstant polynomial. If \( A \in \mathrm{M}_n(R) \) is similar to a matrix in $\mathrm{M}_n(R)$ with all diagonal entries equal to zero, then \( A \in p[ \mathrm{M}_n(R), \mathrm{M}_n(R)] \).
\end{theorem}

\begin{proof}
	Suppose \( A \in \mathrm{M}_n(R) \) is similar to a matrix \( A' \in \mathrm{M}_n(R) \) with all diagonal entries zero, so that \( A = G A' G^{-1} \) for some invertible \( G \in \mathrm{M}_n(R) \). We can express \( A' \) as the difference of two triangular matrices  $A' = L - U,$ where \( L \) is lower triangular and \( U \) is upper triangular, and both have zero diagonal entries. Since \(F\) is infinite and \(p\) is a nonconstant polynomial, Lemma~\ref{distinct} implies that the image \(p(F)\) is infinite. Hence, we may select \(n\) distinct elements \(\beta_1,\dots,\beta_n \in p(F)\). For each \(\beta_i\), choose \(\alpha_i \in F\) such that \(p(\alpha_i) = \beta_i\). If \(\alpha_i = \alpha_j\) for some \(i \ne j\), then \(p(\alpha_i) = p(\alpha_j)\), which contradicts the distinctness of \(\beta_i\) and \(\beta_j\). Thus, both the elements \(\alpha_1, \dots, \alpha_n\) and their images under \(p\) are pairwise distinct. Define 	$$D = \mathrm{diag}(\alpha_1, \ldots, \alpha_n), \quad L_1 = L + p(D) \quad \text{and} \quad U_1 = p(D) + U.$$ By Lemma~\ref{cheo hoa}, there exist invertible matrices \( G_1, G_2 \in \mathrm{M}_n(R) \) such that  
	\[
	G_1^{-1} L_1 G_1 = p(D) = G_2^{-1} U_1 G_2.
	\]  
	Then,
	\[
	\begin{aligned}
		A &= G A' G^{-1} = G(L - U)G^{-1} \\
		&= G(L_1 - U_1)G^{-1} \\
		&= G(G_1 p(D) G_1^{-1} - G_2 p(D) G_2^{-1})G^{-1} \\
		&= p(G G_1 D G_1^{-1} G^{-1}) - p(G G_2 D G_2^{-1} G^{-1}).
	\end{aligned}
	\]  
	Let \( A_1 = G G_1 G_2^{-1} G^{-1} \) and \( B_1 = G G_2 D G_1^{-1} G^{-1} \). Then \[ A = p(A_1B_1) - p(B_1A_1)=p[A_1,B_1], \] as desired.   This completes the proof.
\end{proof}

In what follows, the following corollary, which parallels a result from \cite{Pa_LaWe_93}, follows directly from Theorem~\ref{zero}.  

\begin{corollary}\label{pooo}
	Let \( F \) be an infinite field and \( n > 1 \) an integer. If \( p \in F[x] \) is a nonconstant polynomial, then every noncentral traceless matrix in \( \mathrm{M}_n(F) \) is in \(p[ \mathrm{M}_n(F), \mathrm{M}_n(F)] \). Moreover, if \( F \) has characteristic zero, then $$p[ \mathrm{M}_n(F), \mathrm{M}_n(F)]=\{A\in \mathrm{M}_n(F)\mid\mathrm{trace}(A)=0\}.$$
\end{corollary}

\section{Algebraic structures generated by polynomial commutators}\label{section what}

In this section, we focus on Question~\ref{what}. Understanding representations of matrices described in Theorem~\ref{zero} sheds light on the underlying algebraic structure, particularly in connection with the structure generated  by polynomial commutators. A notable result from \cite[Proposition 3.1]{Pa_BreGaThi_2025} asserts that any triangular matrix of size \( n \geq 3 \) over a ring \( R \) can be expressed as a product of two matrices in \( \mathrm{M}_n(R) \), each having zero entries along the main diagonal. Combining this with Theorem~\ref{zero} leads directly to the following theorem. 

\begin{theorem}\label{al2}
	Let \( R \) be a unital associative algebra over an infinite field \( F \), and let \( n \geq 3 \) be an integer. If \( p\in F[x] \) is a nonconstant polynomial, then every triangular matrix in \( \mathrm{M}_n(R) \) can be written as a product of two elements in $p[\mathrm{M}_n(R),\mathrm{M}_n(R)]$.
\end{theorem}

When the ring $R$ in Theorem~\ref{al2} is restricted to a more perfect algebraic setting, namely, division rings in the noncommutative case and fields in the commutative one, we can go further to show that every matrix is a product of three polynomial commutators. Moreover, when working over fields,  only two such commutators are needed, as presented below.

\begin{theorem}\label{field}
	Let $F$ be an infinite field and let $n>1$ be an integer. If $p\in F[x]$ is a nonconstant polynomial, then every matrix in $\mathrm{M}_n(F)$ can be expressed as a product of two elements in $p[\mathrm{M}_n(F),\mathrm{M}_n(F)]$.
\end{theorem}

\begin{proof}
	It was established in \cite[Theorem 4.1]{Pa_Bo_97} that any matrix over a field can be expressed as the product of two traceless matrices. A closer look at the proof of \cite[Theorem~4.1]{Pa_Bo_97} reveals that both factors in the decomposition are noncentral. Combining \cite[Theorem~4.1]{Pa_Bo_97} with Corollary~\ref{pooo}, it follows that every matrix over a field is a product of two polynomial commutators, as promised.
\end{proof}

The case of matrices over noncommutative division rings presents additional challenges, particularly because it remains unknown whether every noncentral trace-zero matrix over a field is similar to one with all diagonal entries equal to zero. This uncertainty motivates a different strategy in our investigation. Rather than relying on diagonal similarity, we proceed by factoring matrices into products of triangular matrices (or those similar to such), together with Theorem~\ref{al2}.

We begin with a result that proves useful when dealing with invertible matrices over division rings. It may be viewed as a particular instance of a more general theorem from \cite[Theorem 2.1]{Pa_EgGo_19}.

\begin{lemma}[Theorem 2.1, \cite{Pa_EgGo_19}]\label{non singular}  
	Let \( D \) be a division ring and let \( n > 1 \) be an integer. If \( A \in \mathrm{M}_n(D) \) is invertible and noncentral, then there exists an invertible matrix \( P \in \mathrm{M}_n(D) \) such that  $P^{-1} A P = U V,$ where \( U \) is a lower triangular matrix with all diagonal entries equal to 1, and \( V \) is an upper triangular matrix with diagonal entries \( 1, 1, \ldots, 1, d \) for some \( d \in D \setminus \{0\} \).
\end{lemma}

For singular matrices, we rely on a classical block decomposition, which allows us to separate the invertible and nilpotent components. This decomposition follows directly from \cite[Theorem 15, Page 28]{Bo_Ja_43}.

\begin{lemma}[Theorem 15, \cite{Bo_Ja_43}]\label{singular}  
	Let \( D \) be a division ring and \( n > 1 \) an integer. If \( A \in \mathrm{M}_n(D) \) is singular but not nilpotent, then there exists an invertible matrix \( P \in \mathrm{M}_n(D) \) such that  
	\[
	P^{-1} A P = \begin{pmatrix}
		A_1 & 0 \\ 0 & A_2
	\end{pmatrix},
	\]
	in which \( A_1 \in \mathrm{M}_{n-k}(D) \) is invertible and \( A_2 \in \mathrm{M}_k(D) \) is nilpotent for some positive integer \( k \).
\end{lemma}

As for nilpotent matrices, we recall a well-known structural result that every nilpotent matrix over a division ring is similar to a block diagonal matrix composed of Jordan blocks with zero diagonal entries. This is formalized in the next lemma.

\begin{lemma}[Theorem 5, \cite{Pa_Mo_12}]\label{nil}  
	Let \( D \) be a division ring and \( n > 1 \) a positive integer. If \( N \in \mathrm{M}_n(D) \) is nilpotent, then there exists an invertible matrix \( P \in \mathrm{M}_n(D) \) such that  
	\[
	P^{-1} N P = \bigoplus_{i=1}^s J_{m_i}(0),
	\]
	for some positive integers \( m_1, m_2, \dots, m_s \) satisfying \( m_1 + \cdots + m_s = n \). Each block \( J_{m_i}(0) \) denotes the standard Jordan block of size \( m_i \) with zero on the diagonal.
\end{lemma}

These structural lemmas provide a solid foundation for the next stage of our argument, where we formulate and prove the main result, Theorem~\ref{di}.

\begin{theorem}\label{di}
	Let $D$ be a division ring with infinite center $F$ and let $n\geq3$ be an integer. If $p\in F[x]$ is a nonconstant polynomial, then every matrix in $\mathrm{M}_n(D)$ can be expressed as a product of  three elements in $p[\mathrm{M}_n(D),\mathrm{M}_n(D)]$.
\end{theorem}

\begin{proof}
	Let \( A \in \mathrm{M}_n(D) \). If \( A \) is central, then it must be of the form \( A = \lambda \mathrm{I}_n \) for some \( \lambda \in F \), where \( \mathrm{I}_n \) denotes the identity matrix of size \( n \). By Theorem~\ref{field}, such a matrix can be written as a product \( A = BC \) of two polynomial commutators in \( \mathrm{M}_n(F) \). Applying Theorem~\ref{field} again, the matrix \( C \) can itself be expressed as a product of two polynomial commutators in \( \mathrm{M}_n(F) \). Therefore, \( A \) is a product of three polynomial commutators in \( \mathrm{M}_n(D) \).
	
	We now consider the case where \( A \) is noncentral. If \( A \) is invertible, then by Lemma~\ref{non singular}, there exists an invertible matrix \( P \in \mathrm{M}_n(D) \) such that $P^{-1} A P = UV,$
	where \( U \) is a lower triangular matrix with all diagonal entries equal to 1, and \( V \) is an upper triangular matrix with diagonal entries \( 1, 1, \dots, 1, d \), for some \( d \in D \setminus \{0\} \). Since the property of being a product of polynomial commutators is preserved under similarity, it suffices to consider \( A = UV \).
	
	As \( F \) is infinite, we can choose distinct elements \( \alpha_1, \alpha_2, \ldots, \alpha_n \in F \) whose sum is zero, and define \( E = \mathrm{diag}(\alpha_1, \alpha_2, \ldots, \alpha_n) \). Then \( E \) is noncentral, and by Corollary~\ref{pooo}, it is a polynomial commutator in \( \mathrm{M}_n(F) \) relative to \( p \). Define \( U_1 = UE \) and \( V_1 = E^{-1}V \), so that \( A = U_1 V_1 \). By Lemma~\ref{cheo hoa}, \( U_1 \) is similar to \( E \), and hence it is also a polynomial commutator in \( \mathrm{M}_n(D) \) relative to \( p \). By Theorem~\ref{al2}, \( V_1 \) can be written as a product of two polynomial commutators in \( \mathrm{M}_n(D) \) relative to \( p \). Thus, \( A \) is again a product of three polynomial commutators.
	
	Next, suppose that \( A \) is nilpotent. Then, by Lemma~\ref{nil} and Corollary~\ref{pooo}, the desired conclusion holds. Finally, if \( A \) is singular but not nilpotent, then Lemma~\ref{singular} guarantees the existence of an invertible matrix \( P \in \mathrm{M}_n(D) \) such that
	\[
	P^{-1} A P = \begin{pmatrix}
		A_1 & 0 \\
		0 & A_2
	\end{pmatrix},
	\]
	where \( A_1 \in \mathrm{M}_{n-k}(D) \) is invertible and \( A_2 \in \mathrm{M}_k(D) \) is nilpotent for some positive integer \( k \). Applying the previous arguments to \( A_1 \) and \( A_2 \) separately completes the proof.
\end{proof}

An analytic perspective on triangular matrices over algebras is highlighted in \cite[Example~3.4]{Pa_BreGaThi_2025}. Let \( A \) be an AW*-algebra of type \( \mathrm{I}_n \) with \( n \geq 3 \), as defined in \cite[Definition~18.2]{Bo_Be_72}. A classical result in the theory of \( C^* \)-algebras asserts that such an algebra \( A \) is *-isomorphic to \( \mathrm{M}_n(C(X)) \), where \( X \) is an extremally disconnected compact Hausdorff space, and \( C(X) \) denotes the algebra of complex-valued continuous functions on \( X \). Recall that a space is said to be \emph{extremally disconnected} if the closure of every open subset is open.

Given any element \( a \in A \cong \mathrm{M}_n(C(X)) \), a result of Deckard and Pearcy \cite[Theorem~2]{Pa_DePe_63} ensures the existence of a unitary matrix \( u \in \mathrm{M}_n(C(X)) \) such that the conjugate \( u a u^* \) is upper triangular. A more conceptual argument for this fact was later given in \cite[Corollary~6]{Pa_Az_74}. By applying Theorem~\ref{al2} to the triangular matrix \( u a u^* \), we deduce that it is a product of two polynomial commutators. Since the conjugation by a unitary matrix preserves this property (thanks to \( u^{-1} = u^* \) and the closure of polynomial commutators under conjugation), the same holds for \( a \) itself. This yields the following result, Theorem~\ref{AW}.

\begin{theorem}\label{AW}
	Let $A$ be an AW*-algebra of type $\mathrm{I}_n$ with $n\geq3$ and let $p\in\mathbb{C}[x]$ be a nonconstant polynomial. Then, every element of $A$ can be written as a product of two elements in $p[A,A]$.
\end{theorem}

As mentioned in Section~\ref{intro}, we continue Question~\ref{what} with linear spans generated by $p[R,R]$. We recall Question~\ref{que1}, which asks whether the inclusion  
\[
p[R,R] \supseteq \{a \in R \mid \mathrm{trace}(a) = 0\}
\]
holds. From Theorem~\ref{zero}, we recall that for matrix rings, any matrix similar to one with all diagonal entries zero can be written as a polynomial commutator relative to \( p \), and such matrices form a subset of traceless matrices. It is known from \cite{Pa_Al_57, Pa_Sho_36} that in matrix rings over fields, the set of traceless matrices coincides with the additive commutator subgroup, namely $[R, R] = \{a \in R \mid \mathrm{trace}(a) = 0\}.$ This naturally leads to the question of whether \([R, R] \subseteq p[R, R]\). Under milder assumptions, the set \([R, R]\) can be interpreted as $\mathrm{span}\,[R,R]$, the linear span  of additive commutators. This naturally leads to a relaxed version of our earlier question: does the inclusion $\mathrm{span}\,[R,R] \subseteq \mathrm{span}\ p[R, R]$ hold, where \(\mathrm{span}\ p[R, R]\) denotes the linear span of all elements of the form \(p(ab) - p(ba)\)? Interestingly, this version can be addressed by a result from \cite[Corollary 2.6]{Pa_Bre_20}, which we recall below for convenience.

Adapted from \cite[Corollary 2.6]{Pa_Bre_20}, the following result offers a useful perspective. 

\begin{lemma}[Corollary 2.6, \cite{Pa_Bre_20}] \label{key}
	Let \( A \) be an associative simple algebra over an infinite field \( F \), and let \( f \) be a polynomial in noncommuting variables with coefficients in \( F \). If \( f \) is neither a polynomial identity nor a central polynomial of \( A \), then 
	\[
	\mathrm{span}\ [A, A] \subseteq \mathrm{span}\ f(A).
	\]
\end{lemma}

This inclusion will be instrumental in confirming the corollary that follows.

\begin{corollary}\label{simple}
	If $R$ is a unital associative simple algebra over an infinite field $F$ and $p\in F[x]$ is a nonconstant polynomial, then $\mathrm{span}\,[R,R] \subseteq \mathrm{span}\ p[R, R]$.
\end{corollary}

\begin{proof}
	Let \( f = p(x_1x_2) - p(x_2x_1) \) be a polynomial in noncommuting variables \( x_1, x_2 \) with coefficients in \( F \). Observe that \( R \) is a simple \( F \)-algebra and that \( \mathrm{span}\,p[R, R] \) coincides with \( \mathrm{span}\,f(R) \). The conclusion then follows directly from Lemma~\ref{key}.
\end{proof}

Keeping Corollary~\ref{simple} in mind, we now turn to the structural question of whether the matrix algebra can be decomposed as a sum of its center and the subspace generated by polynomial commutators. To approach this, we draw upon the following lemma.

\begin{lemma}\label{dd}
	Let \( D \) be a division ring with infinite center \( F \), and let \( n \) be a positive integer. Suppose that one of the following conditions holds:
	\begin{enumerate}[\rm (i)]
		\item \( D \) is an algebraic division ring of characteristic zero.
		\item \( D \) is finite-dimensional over \( F \), where \( F \) has characteristic \( p > 0 \), and \( p \) does not divide \( \sqrt{\dim_F D} \) or \( n \).
	\end{enumerate} Then,
	\[
	\mathrm{M}_n(D) = Z(\mathrm{M}_n(D)) + [\mathrm{M}_n(D), \mathrm{M}_n(D)] = \{a+b\mid a\in Z(\mathrm{M}_n(D)), b\in [\mathrm{M}_n(D), \mathrm{M}_n(D)]\},
	\]where \( Z(\mathrm{M}_n(D)) \) denotes the center of \( \mathrm{M}_n(D) \). 
\end{lemma}

\begin{proof}
	Recall that a division ring \( D \) is called \emph{algebraic} if every element of \( D \) satisfies a nonzero polynomial with coefficients in the center of \( D \). Part (i) follows directly from \cite[Lemma 5.8]{Pa_DuHaSo_24}. Specifically, \cite[Lemma 5.8]{Pa_DuHaSo_24} shows that if \( D \) is an algebraic division ring of characteristic zero, then the matrix algebra \( \mathrm{M}_n(D) \) decomposes as
	\[
	\mathrm{M}_n(D) = Z(\mathrm{M}_n(D)) + [\mathrm{M}_n(D), \mathrm{M}_n(D)].
	\]
	
	Now, considering part (ii), we turn to the case of positive characteristic, which is addressed in \cite[Theorem 5]{Pa_ChuLee_79}. Since \( p \) does not divide \( \sqrt{\dim_F D} \), we can apply \cite[Theorem 5]{Pa_ChuLee_79}, which asserts $D = F + [D, D].$
	Using this result in conjunction with \cite[Theorems 1 and 3]{Pa_ChuLee_79}, we can then conclude	
	$\mathrm{M}_n(D) = \mathrm{M}_n(F) + [\mathrm{M}_n(D), \mathrm{M}_n(D)].$
	Therefore, for any matrix \( A \in \mathrm{M}_n(D) \), we can express it as the sum of a matrix \( B \in \mathrm{M}_n(F) \) and a matrix \( C \in [\mathrm{M}_n(D), \mathrm{M}_n(D)] \). Since \( p \) does not divide \( n \), the matrix \( A \) can be written as
	\[
	A = \left( B - \frac{1}{n} \mathrm{trace}(B) \mathrm{I}_n \right) + \frac{1}{n} \mathrm{trace}(B) \mathrm{I}_n + C.
	\]	
	Moreover, since the matrix \( B - \frac{1}{n} \mathrm{trace}(B)\mathrm{I}_n \) is traceless, it follows from \cite[Lemma 2]{Pa_ChuLee_79} that
	$B - \frac{1}{n} \mathrm{trace}(B) \mathrm{I}_n \in [\mathrm{M}_n(D), \mathrm{M}_n(D)],$
	which implies that	 $A \in Z(\mathrm{M}_n(D)) + [\mathrm{M}_n(D), \mathrm{M}_n(D)].$
	Thus, we conclude
	$\mathrm{M}_n(D) = Z(\mathrm{M}_n(D)) + [\mathrm{M}_n(D), \mathrm{M}_n(D)].$
	This completes the proof.	
\end{proof}

By combining Lemma~\ref{dd} with Corollary~\ref{simple}, we can directly establish the following theorem (Theorem~\ref{linear p}), so the proof is omitted.

\begin{theorem}\label{linear p}
	Let \( D \) be a division ring with infinite center \( F \), and let \( n \) be a positive integer. The following  holds:
	$\mathrm{M}_n(D) = Z(\mathrm{M}_n(D)) + \mathrm{span}\ p[\mathrm{M}_n(D),\mathrm{M}_n(D)],$ in either of the following cases:
	\begin{enumerate}[\rm (i)]
		\item \( D \) is an algebraic division ring of characteristic zero.
		\item \( D \) is finite-dimensional over \( F \), where \( F \) has characteristic \( p > 0 \), and \( p \) does not divide \( \sqrt{\dim_F D} \) or \( n \).
	\end{enumerate} 	
\end{theorem}

\section{How big can the polynomial commutator of two matrices be?}\label{how big}

As outlined in Section~\ref{intro}, we now turn to Questions~\ref{estimate} and~\ref{estimate1}, which concern the extent of noncommutativity in polynomial expressions involving matrices. Commutators of matrices have been central to developments in linear algebra, operator theory, and quantum mechanics. A classical result in this area is the Böttcher--Wenzel inequality \cite[Theorem 4.1, page 225]{Pa_BoWe_05}, which asserts that for any \( A, B \in \mathrm{M}_n(\mathbb{C}) \), it follows that
\begin{equation}
	\|[A, B]\|_F^2 \leq 2 \|A\|_F^2 \|B\|_F^2.\label{inequality}
\end{equation}
This inequality \eqref{inequality} has inspired a broad body of work investigating commutator-type expressions, including \( q \)-deformations \cite{Pa_ChKiOhSi_22}, higher-order analogues \cite{Pa_CaCeSo_95}, and noncommutative polynomials~\cite{Pa_Se_16}.

In this section, our focus is on polynomial commutators. Our aim is to measure the failure of commutativity in such expressions using three analytic tools: the Frobenius norm $\|\cdot\|_F$, the numerical radius $h(\cdot)$ and the operator norm $\|\cdot\|_{2}$.	

\begin{theorem}
	If $A,B \in \mathrm{M}_n(\mathbb{C})$ and 
	$\displaystyle p(x) = \sum_{k=0}^d a_k x^k \in \mathbb{C}[x]$
	is a nonconstant polynomial  of degree \( d\geq1 \),  then
	\begin{enumerate}[\rm (i)]
		\item $\displaystyle \|p[A,B]\|_F \le \|[A,B]\|_F \sum_{k=1}^d |a_k| \, k \, \|A\|_F^{k-1}\, \|B\|_F^{k-1};$
		\item $\displaystyle h\big(p[A,B]\big) \le h\big([A,B]\big) \sum_{k=1}^d |a_k| \, k \, 2^{2k-1} \, h(A)^{k-1} h(B)^{k-1}$.
	\end{enumerate}
\end{theorem}

\begin{proof}
	From~\eqref{te} we have the standard expansion for the polynomial commutator:
	\[
	p[A,B]
	= \sum_{k=1}^d a_k\big((AB)^k - (BA)^k\big)
	= \sum_{k=1}^d a_k\sum_{j=0}^{k-1} (AB)^j [A,B] (BA)^{k-1-j}.
	\] We begin with the Frobenius norm.  
	Recall that $\|\cdot\|_F$ is submultiplicative,
	\[
	\|XYZ\|_F \le \|X\|_F\,\|Y\|_F\,\|Z\|_F,
	\qquad X,Y,Z\in\mathrm{M}_n(\mathbb{C}),
	\]
	and that
	\[
	\|AB\|_F\le \|A\|_F\|B\|_F,
	\qquad
	\|BA\|_F\le \|A\|_F\|B\|_F.
	\]
	Moreover, these elementary estimates allow us to control each term of the inner sum:
	for every \(0\le j\le k-1\),
	\[
	\|(AB)^j[A,B](BA)^{k-1-j}\|_F
	\le
	\|[A,B]\|_F\,\|A\|_F^{k-1}\|B\|_F^{k-1}.
	\]
	Since there are exactly \(k\) such terms, summing over \(j\) yields
	\[\begin{aligned}
	    \|p[A,B]\|_F
	&\le
	\sum_{k=1}^d |a_k|
	\sum_{j=0}^{k-1}
	\|(AB)^j[A,B](BA)^{k-1-j}\|_F\\
	&\le
	\|[A,B]\|_F
	\sum_{k=1}^d |a_k|\, k\,\|A\|_F^{k-1}\|B\|_F^{k-1}.
	\end{aligned}
	\] We now turn to the numerical radius.  
	Using the inequalities
	\[
	h(X) \le \|X\|_{2} \le 2h(X),\quad X\in\mathrm{M}_n(\mathbb{C}) 
	\quad\text{(see \cite[Theorem 1.3--1, p.~9]{Bo_Gu_97})},
	\]
	together with
	\[
	h(XYZ)\le \|X\|_{2}\; 2h(Y)\; \|Z\|_{2},\qquad X,Y,Z\in\mathrm{M}_n(\mathbb{C}),
	\]
	we obtain, for each \(j\),
	\[
	h\!\big((AB)^j[A,B](BA)^{k-1-j}\big)
	\le
	2^{\,2k-1}\, h([A,B])\, h(A)^{k-1} h(B)^{k-1}.
	\]
	Summing over the \(k\) values of \(j\) now gives
	\[
	\sum_{j=0}^{k-1}
	h\!\big((AB)^j[A,B](BA)^{k-1-j}\big)
	\le
	k\,2^{\,2k-1}\, h([A,B])\,h(A)^{k-1}h(B)^{k-1},
	\]
	and the desired bound for \(h(p[A,B])\) follows immediately.
\end{proof}

Drawing on the probabilistic perspective of~\cite{Pa_Se_16}, note first that the operator norm of \( A \in \mathrm{M}_n(\mathbb{C}) \) measures how much \( A \) can “stretch” a unit vector in the worst case:
\[
\|A\|_{2} = \sup_{\|v\| = 1} \|Av\|.
\]
By contrast, the Frobenius norm reflects an average effect over the entire sphere. In fact, one can show
\[
\|A\|_F^2 = n\,\mathbb{E}_{\|v\| = 1}[\|Av\|^2] = n \int_{\|v\| = 1} \|Av\|^2\, d\sigma(v),
\]
where \( \sigma \) is the uniform probability measure on the unit sphere. Thus, rather than approximating an expectation in a random-matrix model, we have an exact “spherical average.” Before diving into our main results, we will establish the following theorem that derives this very integral formula for \( \|A\|_F \). Although this result may be known, we are not aware of a specific reference. For completeness, we therefore include a brief proof here.

\begin{theorem}\label{mmm}
	If \( A \in \mathrm{M}_n(\mathbb{C}) \), then the Frobenius norm squared of \(A\) is given by
	\[
	\|A\|_F^2 = n \int_{\|v\|=1} \|Av\|^2 \, d\sigma(v),
	\]
	where \( \sigma \) denotes the normalized uniform (Haar) probability measure on the unit sphere \( S^{2n-1} \subseteq \mathbb{C}^n \).
\end{theorem}	

\begin{proof}
	Recall that the Frobenius norm of \(A\) is  defined as
	\[
	\|A\|_F^2 = \mathrm{trace}(A^* A).
	\]  
	For any unit vector \( v \in \mathbb{C}^n \), we can express the squared norm of \( Av \) as  
	\[
	\|Av\|^2 = (Av)^*(Av) = v^* A^* A v,
	\]  
	which is a quadratic form of the Hermitian matrix \(A^* A\). Next, we use the fact that for any scalar \(\alpha\), we have \(\alpha = \mathrm{trace}(\alpha)\), and the cyclic property of the trace, allowing us to rewrite  
	\[
	v^* A^* A v = \mathrm{trace}(v^* A^* A v) = \mathrm{trace}((vv^*) A^* A).
	\]  
	Now, integrating over the unit sphere \( S^{2n-1} \), we get  
	\[
	\int_{\|v\|=1} v^* A^* A v \, d\sigma(v) = \mathrm{trace}\left( \int_{\|v\|=1} (vv^*) \, d\sigma(v) \, A^* A \right).
	\]Due to the invariance of the distribution of \(v\) under unitary transformations \(U \in U(n)\) (see \cite[Proposition 2.2, Page 501]{Pa_Yupa_22}), we conclude that  
	\[
	U \left(\int_{\|v\|=1} vv^* \, d\sigma(v)\right) U^* = \int_{\|v\|=1} (Uv)(Uv)^* \, d\sigma(v) = M,
	\]  
	where \( M \) is the matrix formed by the integral. Since \( M \) is invariant under all unitary transformations, it must be a scalar multiple of the identity matrix, i.e.,  $M = c \mathrm{I}_n$ for some $c\in\mathbb{C}$ (see \cite[page 4279]{Pa_We_89}). Taking the trace of both sides, we have  
	\[
	\mathrm{trace}(M) = \int_{\|v\|=1} \mathrm{trace}(vv^*) \, d\sigma(v) = \int_{\|v\|=1} \|v\|^2 \, d\sigma(v) = 1,
	\]  
	while  $\mathrm{trace}(c\mathrm{I}_n) = cn.$
	Thus, \( cn = 1 \), which gives \( c = \frac{1}{n} \). Therefore,  
	\[
	\int_{\|v\|=1} vv^* \, d\sigma(v) = \frac{1}{n} \mathrm{I}_n.
	\]Substituting this result back into the integral expression for the quadratic form, we obtain  
	\[
	\int_{\|v\|=1} \|Av\|^2 \, d\sigma(v) = \mathrm{trace}\left( \frac{1}{n} \mathrm{I}_n A^* A \right) = \frac{1}{n} \mathrm{trace}(A^* A) = \frac{1}{n} \|A\|_F^2.
	\]  
	Multiplying both sides by \( n \), we arrive at the desired identity  
	\[
	\|A\|_F^2 = n \int_{\|v\|=1} \|Av\|^2 \, d\sigma(v),
	\]  
	which completes the proof.
\end{proof}

Keeping Theorem~\ref{mmm} in mind, we now provide an estimate for the average Frobenius norm bound related to polynomial commutators, as follows.

\begin{theorem}
	Let \( A, B \in \mathrm{M}_n(\mathbb{C}) \), and let
	\[
	p(x) = \sum_{k=0}^d a_k x^k \in \mathbb{C}[x]
	\]
	be a nonconstant polynomial of degree \( d \). Then,
	\[
	\|p[A,B]\|_F^2 \leq n \left( \int_{\|v\|=1} \| [A, B] v\|^2 \, d\sigma(v) \right) \left( \sum_{k=1}^d |a_k|\, k\, \|A\|_2^{k-1} \|B\|_2^{k-1} \right)^2,
	\]
	where the integral captures the average quadratic effect of the commutator \( [A, B] \) over the unit sphere in \( \mathbb{C}^n \).
\end{theorem}

\begin{proof}
	Starting from the expansion \eqref{te}, one knows that  
	\[
	p[A,B]
	=\sum_{k=1}^{d}a_k\;\sum_{j=0}^{k-1}(AB)^{j}\,[A,B]\,(BA)^{\,k-1-j}.
	\]
	Define $T_{k,j} \;=\;(AB)^{j}\,[A,B]\,(BA)^{\,k-1-j}$,
	so that
	\[
	p[A,B]\;=\;\sum_{k=1}^d\sum_{j=0}^{k-1}a_k\,T_{k,j}.
	\]
	Set  $C \;=\; p[A,B]$. Taking Theorem~\ref{mmm} into account, it follows that  
	\[
	\|C\|_F^2
	= n \int_{\|v\|=1} \|C\,v\|^2 \,d\sigma(v),
	\]
	where the integration runs over the unit sphere with normalized surface measure \(\sigma\).  In particular, it is known that
	\[
	C\,v
	=\sum_{k=1}^d\sum_{j=0}^{k-1}a_k\,T_{k,j}v,
	\]
	and the triangle inequality gives
	\[
	\|C\,v\|
	\;\le\;
	\sum_{k=1}^d|a_k|\;\sum_{j=0}^{k-1}\|T_{k,j}v\|.
	\]
	Since each block satisfies
	\[
	\|T_{k,j}v\|
	\;\le\;
	\|(AB)^j\|_2\,\|[A,B]v\|\,\|(BA)^{k-1-j}\|_2
	\;\le\;
	\|A\|_2^{\,k-1}\,\|B\|_2^{\,k-1}\,\|[A,B]v\|,
	\]
	and summing over \(j=0,\dots,k-1\) contributes a factor of \(k\),
	\[
	\|C\,v\|
	\;\le\;
	\left(\sum_{k=1}^d|a_k|\,k\,\|A\|_2^{\,k-1}\|B\|_2^{\,k-1}\right)
	\;\|[A,B]v\|.
	\]Squaring and integrating gives
	\[
	\int_{\|v\|=1}\|C\,v\|^2\,d\sigma(v)
	\;\le\;
	\left(\sum_{k=1}^d|a_k|\,k\,\|A\|_2^{\,k-1}\|B\|_2^{\,k-1}\right)^{\!2}
	\;\int_{\|v\|=1}\|[A,B]v\|^2\,d\sigma(v).
	\]
	Finally, multiplying by \(n\) yields the desired bound:
	\[
	\|p(AB)-p(BA)\|_F^2
	\;\le\;
	n\;\left(\int_{\|v\|=1}\!\|[A,B]v\|^2\,d\sigma(v)\right)
	\left(\sum_{k=1}^d|a_k|\,k\,\|A\|_2^{\,k-1}\|B\|_2^{\,k-1}\right)^{2},
	\]which completes the proof.
\end{proof}


{\small
\begin{thebibliography}{99}

\bibitem{Pa_Aa_18}
M. Aaghabali, S. Akbari, and M. H. Bien.
Division algebras with left algebraic commutators.
\emph{Algebr. Represent. Theory} 21, no. 4, 807--816, 2018.

\bibitem{Pa_Al_57}
A. A. Albert and B. Muckenhoupt.
On matrices of trace zero.
\emph{Michigan Math. J.} 4, 1--3, 1957.

\bibitem{Pa_Am_66}
S. A. Amitsur.
Rational identities and applications to algebra and geometry.
\emph{J. Algebra} 3, 304--359, 1966.

\bibitem{Pa_AmRo_94}
S. Amitsur and L. Rowen.
Elements of reduced trace 0.
\emph{Israel J. Math.} 87, 161--179, 1994.

\bibitem{Pa_Az_74}
E. A. Azoff.
Borel measurability in linear algebra.
\emph{Proc. Amer. Math. Soc.} 42, 346--350, 1974.

\bibitem{Bo_Be_96}
K. I. Beidar, W. S. Martindale, and A. V. Mikhalev.
\emph{Rings with Generalized Identities}.
Monographs and Textbooks in Pure and Applied Mathematics.
Marcel Dekker, 1996.

\bibitem{Bo_Be_72}
S. K. Berberian.
\emph{Baer *-rings}.
Die Grundlehren der mathematischen Wissenschaften, Band 195.
Springer-Verlag, 1972.

\bibitem{Pa_Bo_97}
J. D. Botha.
Products of matrices with prescribed nullities and traces.
\emph{Linear Algebra Appl.} 252, 173--198, 1997.

\bibitem{Pa_BoWe_05}
A. Böttcher and D. Wenzel.
How big can the commutator of two matrices be and how big is it typically?
\emph{Linear Algebra Appl.} 403, 216--228, 2005.

\bibitem{Pa_Bre_20}
M. Brešar.
Commutators and images of noncommutative polynomials.
\emph{Adv. Math.} 374, 107346, 2020.

\bibitem{Pa_BreGaThi_2025}
M. Brešar, E. Gardella, and H. Thiel.
Products of commutators in matrix rings.
\emph{Canad. Math. Bull.} 68, no. 2, 512--529, 2025.

\bibitem{Pa_CaCeSo_95}
M. J. Carro, J. Cerda, and J. Soria.
Higher order commutators in interpolation theory.
\emph{Math. Scand.} 77, 301--319, 1995.

\bibitem{Pa_ChKiOhSi_22}
D. Chruściński, G. Kimura, H. Ohno, and T. Singal.
Bounds for q-deformed commutators.
\emph{Linear Algebra Appl.} 646, 181--205, 2022.

\bibitem{Pa_ChuLee_79}
C. L. Chuang and P. H. Lee.
Idempotents in simple rings.
\emph{J. Algebra} 56, 510--515, 1979.

\bibitem{Pa_DePe_63}
D. Deckard and C. Pearcy.
On matrices over the ring of continuous complex valued functions on a Stonian space.
\emph{Proc. Amer. Math. Soc.} 14, 322--328, 1963.

\bibitem{Pa_DuHaSo_24}
T. H. Dung, B. X. Hai, and T. N. Son.
Reversibility in matrix rings and group algebras.
\emph{Period. Math. Hungar.} 90, 203--216, 2024.

\bibitem{Pa_DuSo_25}
T. H. Dung and T. N. Son.
On Kursov's theorem for matrices over division rings.
\emph{Linear Algebra Appl.} 704, 218--230, 2025.

\bibitem{Pa_EgGo_19}
E. A. Egorchenkova and N. L. Gordeev.
Products of commutators on a general linear group over a division algebra.
\emph{J. Math. Sci.} 243, 561--572, 2019.

\bibitem{Pa_Ga_25}
E. Gardella and H. Thiel.
Rings and C*-algebras generated by commutators.
\emph{J. Algebra} 662, 214--241, 2025.

\bibitem{Bo_Gu_97}
K. E. Gustafson and D. K. M. Rao.
\emph{Numerical Range}.
Springer, 1997.

\bibitem{Pa_Hai.Dung.Bien_2022}
B. X. Hai, T. H. Dung, and M. H. Bien.
Almost subnormal subgroups in division rings with generalized algebraic rational identities.
\emph{J. Algebra Appl.} 21, no. 4, Paper No. 2250075, 2022.

\bibitem{Pa_He_72}
I. N. Herstein and A. Ramer.
A note on division algebras.
\emph{Canad. J. Math.} 24, 734--736, 1972.

\bibitem{Bo_Ja_43}
N. Jacobson.
\emph{The Theory of Rings}.
Amer. Math. Soc., vol. 2, 1943.

\bibitem{Pa_Ka_48}
I. Kaplansky.
Rings with a polynomial identity.
\emph{Bull. Amer. Math. Soc.} 54, 575--580, 1948.

\bibitem{Pa_Ka_51}
I. Kaplansky.
A theorem on division rings.
\emph{Canad. J. Math.} 3, 290--292, 1951.

\bibitem{Pa_LaWe_93}
T. J. Laffey and T. T. West.
Trace-zero matrices and polynomial commutators.
\emph{Irish Math. Soc. Bull.} 31, 11--13, 1993.

\bibitem{Bo_Lam_01}
T. Y. Lam.
\emph{A First Course in Noncommutative Rings}.
Graduate Texts in Mathematics, vol. 131.
Springer, 2001.

\bibitem{Pa_Me_13}
Z. Mesyan.
Polynomials of small degree evaluated on matrices.
\emph{Linear Multilinear Algebra} 61, no. 11, 1487--1495, 2013.

\bibitem{Pa_Mo_12}
A. Mohammadian.
Sums and products of square-zero matrices.
\emph{Comm. Algebra} 40, 4568--4574, 2012.

\bibitem{Pa_Se_16}
D. Serre.
Non-commutative standard polynomials applied to matrices.
\emph{Linear Algebra Appl.} 490, 202--223, 2016.

\bibitem{Pa_Sho_36}
K. Shoda.
Einige Sätze über Matrizen.
\emph{Japan J. Math.} 13, 361--365, 1936.

\bibitem{Pa_We_89}
R. F. Werner.
Quantum states with Einstein-Podolsky-Rosen correlations admitting a hidden-variable model.
\emph{Phys. Rev. A} 40, 4277--4281, 1989.

\bibitem{Pa_Yupa_22}
C. Yuan and P. A. Parrilo.
Maximizing products of linear forms, and the permanent of positive semidefinite matrices.
\emph{Math. Program.} 193, 499--510, 2022.

\bibitem{Pa_Zha_97}
F. Zhang.
Quaternions and matrices of quaternions.
\emph{Linear Algebra Appl.} 251, 21--57, 1997.

\end{thebibliography}
}
\end{document}